\documentclass[12pt,a4paper]{article}
\usepackage{latexsym,amsmath,amssymb,graphicx}

\def\reff#1{(\ref{#1})}
\def\T{^\top}
\def\eps{\varepsilon}
\def\R{{\mathbb R}}

\def\d{\overline d}

\newtheorem{cor}{Corollary}[section]
\newtheorem{prop}{Proposition}[section]
\newtheorem{thm}{Theorem}[section]
\newtheorem{example}{Example}[section]

\newtheorem{lemma}{Lemma}[section]

\newtheorem{rem}{Remark}[section]

\newenvironment{proof}{\noindent {\bf Proof. }}{\hfill $\Box$ \\ \medskip}
\def\reff#1{(\ref{#1})}
\def\qf#1#2{#1\T  #2 #1}

\def\irg#1#2{{ \in\! [{#1}\! : \! {#2}]}}
\def\T{^\top}
\def\a#1{\alpha _{#1}}

\def\DD{{\mathcal D}}

\def\CC{{\mathcal C}}

\def\ker{\mbox{ker }}

\def\tr{\mbox{trace }}

\def\beq#1{\begin{equation}\label{#1}}
\def\eeq{\end{equation}}
\def\bep{\begin{proof}}
\def\ep{\end{proof}}

\def\lk{\left\{}
\def\rk{\right \} }
\def\du#1#2{\langle {#1} , {#2} \rangle}
\def\norm#1{ \| #1 \|}
\def\reff#1{(\ref{#1})}
\def\qf#1#2{#1\T  #2 #1}

\def\T{^\top}
\def\a#1{\alpha _{#1}}

\def\DD{{\mathcal D}}

\def\CC{{\mathcal C}}

\def\ker{\mbox{ker }}

\def\tr{\mbox{trace }}

\def\beq#1{\begin{equation}\label{#1}}
\def\eeq{\end{equation}}
\def\bep{\begin{proof}}
\def\ep{\end{proof}}

\def\lk{\left\{}
\def\rk{\right \} }
\def\du#1#2{\langle {#1} , {#2} \rangle}
\def\norm#1{ \| #1 \|}
\def\lm{\emptyset}
\def\ignore#1{}

\def\bea#1{\begin{array}{#1}}
\def\ea{\end{array}}

\def\conv{{\mbox{conv }}}

\def\DD{{\mathcal K}}

\def\CC{{\mathcal C}}

\def\ker{\mbox{ker }}

\def\tr{\mbox{trace}}

\def\a{{\mathsf a}}
\def\b{{\mathsf b}}

\def\d{{\mathsf d}}

\def\o{{\mathsf o}}

\def\q{{\mathsf q}}
\def\r{{\mathsf r}}
\def\s{{\mathsf s}}

\def\u{{\mathsf u}}
\def\v{{\mathsf v}}

\def\x{{\mathsf x}}
\def\y{{\mathsf y}}
\def\z{{\mathsf z}}

\def\eps{\varepsilon}
\def\Ab{{\mathsf A}}
\def\Bb{{\mathsf B}}

\def\Fb{{\mathsf F}}

\def\Hb{{\mathsf H}}
\def\Ib{{\mathsf I}}
\def\Jb{{\mathsf J}_0}

\def\Mb{{\mathsf M}}

\def\Ob{{\mathsf O}}

\def\Qb{{\mathsf Q}}
\def\Rb{{\mathsf R}}
\def\Sb{{\mathsf S}}

\def\Xb{{\mathsf X}}

\begin{document}
\begin{titlepage}
{\Large \raggedright \sf
Extended Trust--Region Problems with One or Two Balls:\\[0.5em] Exact Copositive and Lagrangian Relaxations}  
\bigskip
\begin{flushleft}
  \normalsize
Immanuel M. Bomze \\ [1.2ex]
ISOR and VCOR, University of Vienna, Austria\\ [2ex]
Vaithilingam Jeyakumar\\ [1.2ex]  School of Mathematics and Statistics, University of New South Wales, Sydney, Australia\\ [2ex]
Guoyin Li\\ [1.2ex]  School of Mathematics and Statistics, University of New South Wales, Sydney, Australia\\ [3ex]
\end{flushleft}
\rule{\textwidth}{0.75pt} \vskip0.4cm
\begin{abstract}
\small
We establish a geometric condition guaranteeing exact copositive relaxation for the nonconvex quadratic optimization problem under two quadratic and several linear constraints, and present sufficient conditions for global optimality in terms of generalized Karush-Kuhn-Tucker multipliers. The copositive relaxation is tighter than the usual Lagrangian relaxation. We illustrate this by providing a whole class of quadratic optimization problems that enjoys exactness of copositive relaxation while the usual Lagrangian duality gap is infinite. Finally, we also provide verifiable conditions under which both
the usual Lagrangian relaxation and the copositive relaxation are exact for an extended CDT (two-ball trust-region) problem. Importantly, the sufficient conditions can be verified by solving linear optimization problems.

\end{abstract}
\bigskip
\bigskip

\centerline{Revised Version: September 15, 2017}
\normalsize

\end{titlepage}
\newpage
\setcounter{page}{1}

\section{Introduction}
Consider the following nonconvex quadratic optimization problem, which is referred to as the extended trust region problem:
\begin{eqnarray*}
 \mathrm{(P)} & \displaystyle \min_{\x \in \mathbb{R}^n} & \x\T\Qb_0 \x+2 \q_0\T\x \\
& \mbox{ subject to } & \x\T\Qb_1 \x +2 \q_1\T\x \le 1  \\
& & {\norm {\Ab\x-\a}}^2 \le 1 \\
& & \Bb\x \le \b,
\end{eqnarray*}
where $\Qb_0,\Qb_1$ are $(n \times n)$ symmetric matrices, $\Ab$ is an $(\ell \times n)$ matrix, $\Bb$ is an $(m \times n)$ matrix, $\a \in \mathbb{R}^\ell$, $\b \in \mathbb{R}^m$
and $\q_0,\q_1 \in \mathbb{R}^n$. Model problems of this form arise from robust optimization problems  under matrix
norm or polyhedral data uncertainty \cite{Beck06,Jeya13a} and the application of the trust region method  \cite{Conn00} for
solving constrained optimization problems, such as nonlinear optimization problems
with nonlinear and linear inequality constraints \cite{Bomz13b,Powell}.  It covers many important and challenging quadratic optimization (QP) problems such as those with box constraints; trust region problems with additional linear constraints; and the CDT (Celis-Dennis-Tapia or two-ball trust-region) problem \cite{AiZh09,Bomz13b,Bure13a,Locatelli,Yang13a}.  In general, with no further structure on the additional linear constraints $\Bb \x\le\b$, the model problem~(P) is NP-hard as it encompasses the quadratic optimization problem with box constraints.

In the special case where $\Qb_1$ is the identity matrix, $\Ab,\Bb$ are zero matrices and $\a,\b,\q_1$ are zero vectors, the model problem~(P) reduces to the well-known trust-region model.
It has been extensively studied from both theoretical and algorithmic points of view \cite{Jeya09a,Yuan90}. The trust-region problem enjoys exact Lagrangian relaxations. Moreover, its solution can be found by solving a dual Lagrangian system or, equivalently, a semidefinite optimization problem (SDP). Unfortunately, these nice features do not continue to hold for the more general  extended trust-region problem~(P); see~\cite{Jeya13a}. In fact, it has been shown that exactness of Lagrangian (or SDP) relaxation can fail for the CDT problem, or for the trust region problem with only one additional linear inequality constraint.

Recently, copositive optimization has emerged as one of the important tools for studying nonconvex quadratic optimization problems. Copositive optimization is a special
case of convex conic optimization (namely, to minimize a linear function over a cone subject to linear constraints). By now, equivalent copositive reformulations for many important problems are known, among them (non-convex, mixed-binary, fractional) quadratic optimization problems under a mild assumption~\cite{Amar15,Amar11,Bure09a}, and some special optimization problems under uncertainty \cite{Arde16a,Hana16a,Nata11,Xu16a}. In particular, it has been shown in \cite{Bomz13c} that, for quadratic optimization problems with additional nonnegative constraints, copositive relaxations (and its tractable approximations) provides a tighter bound than the usual Lagrangian relaxation.
On the other hand, the techniques in \cite{Bomz13c} are not directly applicable because our model problem does not require the variables to be nonnegative.

In light of rapid evolution of this field, in this paper, we introduce a new copositive relaxation for the extended trust region problem~(P), and present two significant contributions to copositive optimization:
\begin{itemize}
\item
We establish a geometric condition guaranteeing exact copositive relaxation for the nonconvex quadratic optimization problem~(P). We also  present sufficient conditions for global optimality in terms of generalized Karush-Kuhn-Tucker multipliers extending the global optimality conditions obtained for CDT problems \cite{Bomz13b}.
Moreover, we provide a class of quadratic optimization problems that enjoys exactness of the copositive relaxation while the usual Lagrangian duals for these problems yield trivial lower bounds with infinite gaps.

\item In the special case, where (P) is an extended CDT (or two-ball trust region, TTR) problems, we also
derive simple verifiable sufficient conditions, which is independent of the geometric conditions, ensuring both exact copositive relaxation and exact Lagrangian relaxations. In particular, the sufficient conditions can be checked by solving a linear optimization problem.
\end{itemize}

The paper is organized as follows: In Section 2, we first recall notation and terminology, and present some basic facts on copositivity. In Section 3, we introduce
the copositive relaxation for~(P) and its semi-Lagrangian reformulation. We also provide a global optimality condition and prove an
exactness result for this relaxation. In Section 4, we examine the extended CDT problem and provide simple conditions ensuring the tightness of both the copositive relaxation and the usual Lagrangian relaxation. In the appendix, we provide details on how copositive relaxation problems can be approximated by hierarchies of semidefinite and/or linear optimization problems.

\section{Preliminaries}

We abbreviate by $[m\! : \! n]:=\lk m, m+1, \ldots , n\rk$ the integer range between two integers $m,n$ with $m\le n$.
By bold-faced lower-case letters we denote vectors in $n$-dimensional Euclidean space $\R^n$, by bold-faced upper case letters matrices, and by $\T$ transposition. The positive orthant is denoted by
$\R ^n_+ :=\lk \x\in \R ^n: x_i \ge 0\mbox{ for all }i\irg 1n\rk$.
$\Ib_n$ is the $n\times n$ identity matrix.
The letters $\o$ and $\Ob$ stand for zero vectors, and zero matrices, respectively, of appropriate orders.  The set of all $n\times n$ matrices is denoted by $\R^{n\times n}$, and the closure (resp. interior) of a set $S\subset \R^n$ by ${\rm cl}(S)$ (resp. ${\rm int}~S$).

For a given symmetric matrix $\Hb=\Hb\T$, we denote the fact that $\Hb $ is positive-semidefinite by $\Hb \succeq \Ob $.
Sometimes we write instead "$\Hb$ is psd." Denoting the smallest eigenvalue of any symmetric matrix $\Mb=\Mb\T$ by $\lambda_{\rm min}(\Mb)$, we thus have $\Hb \succeq \Ob $ if and only if
$\lambda_{\rm min}(\Hb)\ge 0$. Linear forms in symmetric matrices $\Xb$ will play an important role in this paper; they are expressed by Frobenius duality $\du \Sb\Xb  = \tr (\Sb\Xb )$, where $\Sb=\Sb\T$ is another symmetric matrix of the same order as $\Xb$.
By $\Ab\oplus \Bb$ we denote the direct sum of two square matrices:
\beq{defjo}\Ab\oplus \Bb = \left[\bea{cc} \Ab &\Ob \\ \Ob\T &\Bb \ea\right ]\, ,\quad\mbox{and in particular we will use}\quad \Jb := 1\oplus \Ob = \left[ \bea{cc} 1  &\o\T \\ \o &\Ob \ea\right ]\, .\eeq
For any optimization problem, say (Q), we denote by val(Q) its optimal objective value (attained or not). Consider a quadratic function $q(\x) = \qf\x\Hb - 2\d\T\x + \gamma$  defined on $\R^{n}$, with $q(\o)=\gamma$, $\nabla q(\o)=-2\d$ and $D^2 q(\o)= 2\Hb$ (the factors $2$ being here just for ease of later notation). For this $q$ we define the {\em Shor relaxation matrix~\cite{Shor87}} as
\beq{defm}\Mb(q) :=  \left[\bea{cc}  \gamma  &-\d\T \\ -\d &\Hb \ea\right ]\, .\eeq
Then  $q(\x) \ge 0$ for all $\x \in \R^{n}$ if and only if $\Mb(q) \succeq \Ob$.

Given any cone $\CC$ of symmetric $n\times n$ matrices,
$$\CC^\star :=\lk \Sb = \Sb \T  \in \R^{n\times n } : \du \Sb\Xb \ge 0 \mbox{ for all }\Xb\in \CC\rk $$
denotes the dual cone of $\CC$. For instance, if $\CC = \lk \Xb=\Xb\T  \in \R^{n\times n } : \Xb\succeq \Ob\rk$, then $\CC^\star = \CC$ itself, an example of a {\em self-dual cone}. Trusting the sharp eyes of our readers, we chose a notation with subtle differences between the five-star denoting a dual cone, e.g., $\CC^\star$, and the six-star, e.g. $z^*$, denoting optimality.

The key notion used below is that of {\em copositivity}. Given a symmetric $n\times n$ matrix $\Qb$, and a closed, convex cone $\Gamma \subseteq\R^n$,
we say that
$$\bea{rl} \Qb \mbox{ is }\Gamma\mbox{-copositive if } &\qf\v\Qb \ge 0\mbox{ for all }\v\in \Gamma
\, , \quad\mbox{and that }\\[0,3em]
\Qb \mbox{ is strictly } \Gamma\mbox{-copositive if } &\qf\v\Qb > 0\mbox{ for all }\v\in \Gamma\setminus\lk\o\rk\, .\ea$$
Strict copositivity generalizes positive-definiteness (all eigenvalues strictly positive) and copositivity generalizes positive-semidefiniteness (no eigenvalue strictly negative) of a symmetric matrix. Checking copositivity is NP-hard for most cones $\Gamma$ of interest,  see~\cite{Dick11b,Murt87} for the classical case $\Gamma=\R^n_+$ studied already by Motzkin~\cite{Motz52b} who coined the notion back in 1952. In the sequel, we will use "copositive" synonymous for "$\R^n_+$-copositive" in Motzkin's sense.

The set of all $\Gamma$-copositive matrices forms a closed, convex matrix cone, the {\em copositive cone}
$$\CC^\star _\Gamma:=\lk \Qb =\Qb\T   \in \R^{n\times n}  : \Qb\mbox{ is }\Gamma\mbox{-copositive}\rk$$
with non-empty interior  ${\rm int}~\CC^\star_\Gamma$, which exactly consists of all strictly $\Gamma$-copositive matrices.
However, the cone $\CC^\star_\Gamma$ is not self-dual. Rather one can show that $\CC^\star_\Gamma$ is the dual cone of
$$\CC_\Gamma := \lk \Xb = \Fb\Fb\T : \Fb \mbox{ has } {n+1\choose 2} \mbox{ columns in }\Gamma\rk\, , $$
the cone of $\Gamma$-{\em completely positive (cp)} matrices. Note that the factor matrix $\Fb$ has many more columns than rows.  A perhaps more amenable representation is
$$\CC_\Gamma = \conv \lk \x\x\T : \x\in \Gamma\rk\, , $$
where $\conv S$ stands for the convex hull of a set $S\subset\R^n$. Caratheodory's theorem then elucidates the bound ${n+1\choose 2}$ on the number of columns in $\Fb$ above, which is not sharp in the classical case $\Gamma=\R^n_+$ but asymptotically tight~\cite{Bomz14b,Shak13b}.

Next, we specify a result on reducing $\Upsilon$-copositivity with $\Upsilon=\R^p_+\times \R^{n}$ to a combination of psd and classical copositivity conditions. This
result will be used later on.

\begin{lemma}\label{schur}
Let $\Upsilon=\R^p_+\times \R^{n}$ and partition a $(p+n)\times (p+n)$ matrix $\Mb$ as follows:
$$\Mb = \left [\bea{rl} \Rb &\Sb\T \\ \Sb &\Hb\ea\right ]\quad \mbox{where}\quad \Rb\mbox{ is a } p\times p\mbox{-matrix.}$$
Then $\Mb $ is $\Upsilon$-copositive if and only if the following two conditions hold:
\begin{description}
\item[(a)] $\Hb$ is positive semidefinite and $\Hb\Hb^\dag \Sb=\Sb$, i.e., $\ker \Hb \subseteq \ker \Sb\T$;
\item[(b)] $\Rb-\Sb\T \Hb^\dag \Sb$ is $(\R^p_+-)$copositive.
\end{description}
Here $\Hb^\dag$ is the Moore-Penrose pseudoinverse of $\Hb$.
\end{lemma}
\bep The argument is an easy extension of the arguments that led to~\cite[Thm.3.1]{Bomz13b}. \ep


\section{Relaxations for extended trust region problems}\label{semi}

The problem we study here is given by
\begin{eqnarray*}
 \mathrm{(P)} & \displaystyle \min_{\x \in \mathbb{R}^n} & f_0(\x) := \x\T\Qb_0 \x+2 \q_0\T\x \\
& \mbox{ subject to } & f_1(\x):= \x\T\Qb_1 \x +2 \q_1\T\x -1 \le 0  \\
& & f_2(\x):= {\norm {\Ab\x-\a}}^2 -1 \le 0 \\
& & \Bb\x \le \b.
\end{eqnarray*}
Throughout this paper, we assume that the feasible set of  problem~(P) is non-empty.
The model problem (P) can be reformulated as
\beq{qqpp}
z^*:= \inf\lk f_0(\x) : \x\in F\cap P\rk\quad\mbox{ with }P:=\lk \x\in \R^n : \Bb\x \le \b\rk\, ,
\eeq
where $F:=\lk \x \in \R^n: f_i(\x) \le 0 \, , \, i=1,2\rk$, $\b\in \R^p$ and $\Bb$ is a $p\times n$ matrix.

For our approach, it will be convenient to introduce slack variables $s_j:= b_j - (\Bb\x)_j$ for all $j\irg 1p$, arriving at new primal-feasible points $\y= (\s, \x)\in \R^p_+\times \R^n$, in other words, to replace $P$ with
$$\bar P :=\lk \y= (\s, \x)\in \R^p_+\times \R^n : \bar\Bb \y = \b\rk$$
with the $p\times (p+n)$-matrix $\bar\Bb := [\Ib_p \,| \, \Bb ]$.

We now need to extend all original functions in the obvious way, namely $\bar f_i (\y) =\bar f_i(\s,\x) = f_i(\x)$ by writing $\bar\Qb _i =\Ob\oplus \Qb_i$, i.e., adding $p$ zero rows and $p$ zero columns to $\Qb_i$, arriving at symmetric matrices of order $p+n$; likewise we define $\bar \q_i \T = [\o\T,\q_i\T]$. Finally, by introducing another quadratic constraint, defining $\bar \Qb_{3}= \bar\Bb\T\bar\Bb$, $\bar \q_{3}= \bar\Bb\T\b$ and $c_{3}=\b\T\b$, we rephrase the $p$ linear constraints $\bar\Bb\y=\b$ into one quadratic constraint $\bar f_{3}(\y)={\norm{\bar\Bb\y- \b}}^2 { \le 0}$.

In this way, the original problem~\reff{qqpp} is rephrased in a somehow standardized form, namely
\beq{stand}z^* = \inf\lk \bar f_0(\y) : \bar f_i(\y) \le 0\, , \, i\irg 13\, , \, \y = (\s,\x)\in \R^p_+\times \R^n\rk\, .\eeq

The optimal value $z^*$  of~\reff{qqpp} need not be attained, and it could also be equal to $-\infty$ (in the unbounded case) or to $+\infty$ (in the infeasible case).
Considering $\Qb_i=\Ob$ would also allow for linear inequality constraints. But it is often advisable to discriminate the functional form of constraints, and we will adhere to this principle in what follows.

\subsection*{Copositive relaxation}
Next, we introduce a copositive relaxation for (P). Let $\y=(\s,\x) \in \mathbb{R}^p \times \mathbb{R}^n$. Now consider multipliers $\u\in \R^{3}_+$ of the inequality constraints $\bar f_i(\y)= f_i(\x)\le 0$, $i\irg 13$, and $\v\in \R^p_+$ for the sign constraints $\s\in \R ^p_+$.
 Then we define the full Lagrangian function for problem~\reff{stand} as
$$L(\y;\u,\v) := \bar f_0(\y) + u_1 \bar f_1(\y)+u_2 \bar f_2(\y)+u_3 \bar f_3(\y) - \v\T\s\, .$$
Let $\Upsilon  = \R^{p+1}_+\times \R^n$. Recall that the matrix $\Jb$ and the Shor relaxation matrix $\Mb(q)$ for a quadratic function $q$ are given as in \eqref{defjo} and \eqref{defm} respectively. Then the matrix $\Mb(L (\cdot;\u,\o ))-\mu\Jb$ can be written as below:
\beq{explsl}
\left[\bea{ccc}
-u_1-u_2 + u_{3}{\norm \b}^2-\mu &- u_{3}\b\T & \q_0\T+u_1 \q_1\T-u_2 \a\T \Ab - u_{3}\b\T\Bb \\
- u_{3}\b & u_{3}\Ib_p & u_{3}\Bb \\
\q_0+u_1 \q_1-u_2  \Ab\T \a - u_{3}\Bb\T \b  & u_{3}\Bb\T & \Hb_u + u_{3}\Bb\T\Bb
\ea
\right ]
\eeq
where $\Hb_u = \Qb_0 + u_1 \Qb_1+ u_2 \Ab\T\Ab$ is the Hessian of the Lagrangian function.
We now associate a copositive relaxation for (P)  as follows:
\beq{copd}
{\rm (COP)} \quad z_{\rm COP}^*:= \sup\lk \mu : (\mu,\u )\in \R\times \R^{3}_+  , \; \Mb(L (\cdot;\u,\o ))-\mu\Jb \mbox{ is }\Upsilon\mbox{-copositive}\rk\, ,\eeq
It is worth noting that, unlike in~\cite{Bomz13c}, our model problem does not require the variables to be nonnegative, and so the techniques in constructing a copositive relaxation as in \cite{Bomz13c} cannot be applied directly. Here we achieve this task by introducing nonnegative slack variables.

An important observation  is that the copositive relaxation can be equivalently reformulated as a semi-Lagrangian dual
problem of the problem (P).
Recall that the usual Lagrangian dual (or Lagrangian relaxation) of (P) is given by
\beq{defzfulld}
z_{\rm LD}^* := \sup\lk \Theta (\u,\v) : (\u,\v)\in \R^{3}_+\times \R^p_+\rk\, ,\eeq
where $\Theta (\u,\v) :=\inf \lk L (\y;\u,\v) : \y\in \R^{p+n}\rk$. A form of partial Lagrangian relaxation called semi-Lagrangian of (P) (see \cite{Bomz13c} and the references therein) is given by
\beq{defzsemi}
z_{\rm semi}^* := \sup\lk \Theta_{\rm semi}(\u): \u\in \R^{3}_+\rk\, .\eeq
where $\Theta_{\rm semi}(\u) :=\inf \lk L (\y;\u,\o) : \y\in \R^p_+\times \R^n\rk$. The relation between copositive relaxation, full and semi-Lagrangian bounds can be summarized in the following chain of inequalities
\beq{eq:chain}z_{\rm LD}^*\le z_{\rm COP}^*=z_{\rm semi}^* \le z^*.\, \eeq We note that the relation $z_{\rm LD}^*\le z_{\rm semi}^* \le z^*$ follows by the construction, and the equality
$z_{\rm COP}^*=z_{\rm semi}^*$ follows by adapting the techniques in \cite[Lemma 2.1]{Bomz13c} to the polyhedral cone $\R^p_+\times \R^n$ (see also \eqref{eq:00p} later for a detailed proof).


We now illustrate that, in general, a copositive relaxation can provide a much tighter bound for the model problem (P) than the usual Lagrangian dual.
Indeed, in the following example, we see that the copositive relaxation is tight while the usual
Lagrangian dual yields a trivial lower bound which has infinite gap. As we will see later (Proposition \ref{remark:3.2}), one can indeed construct a whole class of
quadratic optimization problem with exact copositive relaxation but infinite Lagrangian duality gap.
\begin{example}{\bf (Copositive relaxation vs Lagrangian relaxation)} \label{ex:5.1} Consider the following nonconvex quadratic optimization problem with simple linear inequality constraints
\begin{eqnarray*}
& \min & q(\x):= \frac{1}{2}x_1^2+2x_1x_2+ x_2^2 \\
& \mbox{ s.t. } & x_1 \ge 0, \ x_2 \ge 0.
\end{eqnarray*}
Clearly, the objective function is not convex and the optimal value of this problem is $z^*=0$.
We next observe that this problem can be converted to our standard form as
\begin{eqnarray*}
& \min\limits_{(\x,\s)\in \mathbb{R}^2 \times \mathbb{R}^2} & \frac{1}{2}x_1^2+2x_1x_2+ x_2^2 \\
& \mbox{ s.t. } & (x_1-s_1)^2+(x_2-s_2)^2 \le 0 \\
& & s_1 \ge 0, s_2 \ge 0.
\end{eqnarray*}
Then the copositive relaxation reads
\[
z_{\rm COP}^*=\sup_{\mu \in \R, u \ge 0}\lk  \mu : \left[\bea{ccccc}
 -\mu & 0& 0&0 &0 \\
0 & u & 0 & -u & 0 \\
0 & 0 & u & 0 & -u \\
0 & -u & 0 & \frac{1}{2}+u & 1 \\
0 & 0 & -u & 1 & 1+u
\ea
\right ] \;  \mbox{ is }(\mathbb{R}^3_+ \times \mathbb{R}^2) \mbox{-copositive}\rk
\]
Clearly, from the copositivity requirement, $z_{\rm COP}^* \le 0$. Moreover,
it can be verified from Lemma~\ref{schur} that, for $\mu=0$ and $u=1$, the matrix
\[
\left[\bea{ccccc}
0 & 0& 0&0 &0 \\
0 & 1 & 0 & -1 & 0 \\
0 & 0 & 1 & 0 & -1 \\
0 & -1 & 0 & \frac{3}{2} & 1 \\
0 & 0 & -1 & 1 & 2
\ea
\right ] \;  \mbox{ is }(\mathbb{R}^3_+ \times \mathbb{R}^2) \mbox{-copositive.}
\]
Thus, $z_{\rm COP}^*=z^*=0$.

Next we show that $z_{\rm LD}^*=-\infty$. To see this, we only need to show that for each fixed $u \ge 0$ and $\v=(v_1,v_2)\T \in \R_+^2$, we have
\[
\inf_{(\x,\s) \in \mathbb{R}^2 \times \mathbb{R}^2}\{  [\frac{1}{2}x_1^2+2x_1x_2+ x_2^2] + u [(-x_1+s_1)^2+(-x_2+s_1)^2]+v_1 s_1+v_2s_2\}=-\infty.
\]
Indeed, taking $\x=(-t,t)$ and $\s=(t,-t)$ we see that, as $t\rightarrow +\infty$ ,
\begin{eqnarray*}
& & [\frac{1}{2}x_1^2+2x_1x_2+ x_2^2] + u [(-x_1+s_1)^2+(-x_2+s_1)^2]+v_1 s_1+v_2s_2\\
&=& -\frac{1}{2} t^2 +t(v_1-v_2) \rightarrow -\infty\, .
\end{eqnarray*}
\end{example}

\vspace{-0.2cm}

\section{Tightness of copositive relaxation}\label{semitight}

We consider, for $\y=(\s,\x)\in \R^p_+\times \R^n$,  the full Lagrangian function
$$ L(\y; \u,\v) = \bar f_0(\y)+\sum_{i=1}^3 u_i \bar f_i(\y) - \v\T\s, \,  \quad (\u,\v)\in \R^{3}_+\times \R^p\, .$$
As in~\cite{Bomz13c}, let us say that the pair
$( \x ;\u,\v)\in (F\cap P)\times \R^{3}_+\times \R^p$ is a {\em generalized KKT pair} for~\reff{qqpp} if and only if,  for $\s=\b-\Bb \x$ and $\y=(\s,\x)$, it satisfies both the first-order conditions
$\nabla_\y   L(\y; \u, \v) = \o$ and as well the complementarity conditions
$v_k s_k =0$ for all $k\irg 1p$  and $u_i \bar f_i(\y)=0$ for all $i\irg 1{3}$, but without requiring $v_k\ge 0$.

\vspace{-0.2cm}
\subsection*{Geometric conditions for exact copositive relaxation}
Next, we provide a geometric condition ensuring the exactness of the copositive relaxation which
does not rely on the information of KKT pairs. To do this,  denote $\y=(\s,\x)$ and let $\bar{f}_0(\y)=\x\T\Qb_0\x+2 \q_0\T\x$, $\bar{f}_1(\y) = \x\T\Qb_1x+2\q_1\T \x-1$,  $\bar{f}_2(\y) = {\norm {\Ab\x-\a}}^2-1$  and
$\bar{f}_3(\y)=\|\Bb\x+\s-\b\|^2$.
\begin{thm}
For the extended trust region problem {\rm(P)}, let  $$\Omega:=\{[\bar{f}_0(\y),\bar{f}_1(\y),\bar{f}_2(\y),\bar{f}_3(\y)]\T: \y \in  \mathbb{R}^p_+ \times 	 \mathbb{R}^n \}+\mathbb{R}^{4}_+.$$ Suppose that
$\Omega$ is closed and convex.
Then we have $z_{\rm COP}^*=z^*$.
\end{thm}
\begin{proof} Let $z_{\rm semi}^*$ denote the optimal value of the semi-Lagrangian dual \eqref{defzsemi}. We first observe that $z_{\rm COP}^*=z_{\rm semi}^*$. To see this, for any $\mu\in \R$ and any quadratic function $q$ defined on $\mathbb{R}^{p+n}$, it can be directly verified
that
the following two conditions are equivalent:
\begin{description}
\item[(a)] $q(\y)\ge \mu$ for all $\y\in \R^p_+\times \R^n$;
\item[(b)] the $(n+p+1)\times (n+p+1)$-matrix $\Mb(q-\mu)=\Mb(q)-\mu\Jb $  is $\Upsilon$-copositive.
\end{description}
This equivalence implies the identity
\beq{keycor}
\inf\lk q(\y) : \y\in \R^p_+\times \R^n\rk = \sup\lk \mu\in \R : \Mb(q)-\mu\Jb \mbox{ is }\Upsilon\mbox{-copositive}\rk\, .
\eeq
Note that above equality holds, by the usual convention that $\sup\lm = -\infty$, also if $q$ is unbounded from below on $\R^p_+\times\R^n$.
Applying \eqref{keycor} with $q=L (\cdot;\u,\o)$, we see that
\[ \Theta_{\rm semi} (\u ) = \sup\lk \mu : \mu\in \R \, , \; \Mb(L (\cdot;\u,\o ))-\mu\Jb \mbox{ is }\Upsilon\mbox{-copositive}\rk \]
Then it follows from  the definitions of semi-Lagrangian dual and copositive relaxation that
\begin{eqnarray}\label{eq:00p}
z_{\rm semi}^*&=&\sup\lk \Theta_{\rm semi}(\u): \u\in \R^{3}_+\rk \nonumber \\
&=&  \sup\lk \mu : (\mu,u)\in \R \times \R^3_+ \, , \; \Mb(L (\cdot;\u,\o ))-\mu\Jb \mbox{ is }\Upsilon\mbox{-copositive}\rk \nonumber \\
& =& z_{\rm COP}^*.
\end{eqnarray}

As $z^* \ge z_{\rm semi}^*$, we see that $z^* \ge z_{\rm COP}^*$ always holds.  So, we can assume without loss of generality that $z^*>-\infty$. As the feasible set of (P) is nonempty, we have $z^*<+\infty$,
and hence $z^* \in \mathbb{R}$.
Let $\epsilon>0$. Thus $[z^*-\epsilon,0,0,0]\T \notin \Omega$. By the strict separation theorem, there exists $(\mu_0,\mu_1,\mu_2,\mu_3) \neq (0,0,0,0)$
such that
\[
\sum_{i=0}^3 \mu_i a_i >  \mu_0(z^*-\epsilon) \quad\mbox{ for all } \a\in \Omega\, .
\]
As $\Omega+\mathbb{R}^4_+ \subseteq \Omega$, we get $\mu_i \ge 0$ for all $i\irg 03$. Moreover, by the feasibility, we see that $\mu_0>0$. Thus, by dividing
$\mu_0$ on both sides, we see that for all $\y=(\s,\x) \in  \mathbb{R}^p_+ \times \mathbb{R}^n$
\[
 \bar{f}_0(\y)+ \sum_{i=1}^3 \mu_i \bar{f}_i(\y) >z^*-\epsilon,
\]
where $\lambda_i=\mu_i /\mu_0$, $i=1,2,3$. This implies that
\[
z^*-\epsilon\le \inf_{\y \in \mathbb{R}^p_+ \times \mathbb{R}^n}\{\bar{f}_0(\y)+ \sum_{i=1}^3 \mu_i \bar{f}_i(\y)\} \le z_{\rm semi}^*=z_{\rm COP}^*,
\]
where the second inequality follows from the definition of semi-Lagrangian dual \eqref{defzsemi}.
By letting $\epsilon \searrow 0$, we have $z^* \le z_{\rm COP}^*$. As the reverse inequality always holds, the conclusion follows.
\end{proof}

 Before we provide simple sufficient conditions ensuring this
geometrical condition, we will  illustrate it using our previous example.
\begin{example}
Consider the same example as in Example \ref{ex:5.1}. We observe that, in this case $\Ab,\Qb_1$ are zero matrices and
 $\q_1,\a$ are zero vectors, and so, the set $\Omega$ becomes
$$
\Omega:=\{ \left[\bea{c}\frac{x_1^2}2+2x_1x_2+ x_2^2 \\-1\\-1\\(x_1-s_1)^2+(x_2-s_2)^2\ea \right]: (\s,\x) \in  \mathbb{R}^2_+ \times 	\mathbb{R}^2 \} +\mathbb{R}^{4}_+\, .
$$
Then
\[
\Omega=\{[z_1,z_2,z_3,z_4]\T: [z_1,z_4]\T \in \Omega_1, z_2 \ge -1, z_3 \ge -1\},
\]
where
\[
\Omega_1= \{\left[\bea{c}\frac{x_1^2}2+2x_1x_2+ x_2^2 \\(x_1-s_1)^2+(x_2-s_2)^2\ea \right] :(\s,\x) \in  \mathbb{R}^2_+ \times 	\mathbb{R}^2 \}+\mathbb{R}^{2}_+ \, .
\]
Now we provide an analytic expression for $\Omega_1$. Note that, if $x_1=0$, then $[\frac{x_1^2}2+2x_1x_2+ x_2^2,(x_1-s_1)^2+(x_2-s_2)^2 ]\T \in \mathbb{R}_+^2$ and $\mathbb{R}^2_+ \subseteq \Omega_1$ (take $x_2=s_2\ge 0$ and $s_1\ge 0=x_1$ to get an arbitrary point $(x_2^2,s_1^2)\T\in \R^2_+$). Thus we only
need to consider the case where $x_1 \neq 0$. Then
 \begin{eqnarray*}
 \Omega_1 &=& \{\left[\bea{c}\frac{x_1^2}2+2x_1x_2+ x_2^2\\(x_1-s_1)^2+(x_2-s_2)^2\ea \right]:\s \in  \mathbb{R}^2_+, \alpha \in \mathbb{R}, \,  \x=\left[\bea{c} t \\ \alpha t\ea \right] \in \mathbb{R}^2\} +\mathbb{R}^{2}_+ \\
 & = &\{ \left[\bea{c}(\frac{1}{2}+2\alpha+ \alpha^2)t^2 \\\min\{t,0\}^2+\min\{\alpha t,0\}^2\ea \right] :\left[\bea{c} t \\ \alpha\ea \right] \in \mathbb{R}^2 \}+\mathbb{R}^{2}_+ \, ,
 \end{eqnarray*}
 where the last equality follows by noting that $\min_{s \ge 0}(x-s)^2=\min\{x,0\}^2$.
Direct verification now shows that
\[
\Omega_1=\{[a_1,a_2]\T: a_2 \ge -a_1 \ge 0\} \cup \mathbb{R}^2_+\, ,
\]
which is closed and convex. Therefore, $\Omega$ is also closed and convex.
\end{example}

 Next, we provide some verifiable sufficient conditions guaranteeing convexity as well as closedness of $\Omega$. To do this, recall
 that an $n\times n$ matrix $\Mb$ is called a $Z$-matrix if its off-diagonal elements $M_{ij}$ with $1 \le i,j \le n$ and $i \neq j$, are all non-positive. We also need the following joint-range convexity for $Z$-matrices.

 \begin{lemma}
 Let $\Mb_i$, $i\irg 1q$, be symmetric $Z$-matrices of order $n$. Then
 \[
 \{(\x\T \Mb_1 \x , \ldots,\x\T \Mb_q \x): \x \in \mathbb{R}^n\}+ \mathbb{R}^q_+
 \]
is a convex cone.
 \end{lemma}
\begin{proof}
The proof is similar to \cite[Theorem 5.1]{Jeya09a}.
\end{proof}

 \begin{prop}
 Suppose that $\Qb_0,\Qb_1$, $\Ab\T\Ab$ are all $Z$-matrices, $\Bb=-\Ib_n$ and $\q_0,\q_1,\a,\b$ are zero vectors. Then $\Omega$ is convex.
 \end{prop}
 \begin{proof}
Let  $\bar{h}_0(\y)=\x\T\Qb_0\x$, $\bar{h}_1(\y) = \x\T\Qb_1\x$,  $\bar{h}_2(\y) = {\norm {\Ab\x}}^2$  and
 $\bar{h}_3(\y)=\|\x-\s\|^2$ with $\y=(\s,\x)$, so $p=n$ here.
 We first note that $ \Omega=(0,-1,-1,0)+\bar \Omega$
 where
 \[
 \bar \Omega=\{(\bar{h}_0(\y),\bar{h}_1(\y),\bar{h}_2(\y),\bar{h}_3(\y)): \y \in  \mathbb{R}^p_+ \times 	\mathbb{R}^n \}+\mathbb{R}_+^4.
 \]
 To see the convexity of $\Omega$, it suffices to show that $\bar \Omega$
 is convex.
 To verify this, take  $(u_{0},u_{1},u_{2},u_{3}) \in \bar \Omega$ and $(v_{0},v_{1},v_{2},v_{3}) \in \bar \Omega$, and let $\lambda \in [0,1]$.
 Then there exist $(\hat{\s},\hat{\x}) \in \mathbb{R}^n_+ \times \mathbb{R}^n $ and $(\tilde{\s},\tilde{\x}) \in \mathbb{R}^n_+ \times \mathbb{R}^n $ such that
 \[
 \bar h_i(\hat{\s},\hat{\x}) \le u_i \mbox{ and } \bar h_i(\tilde{\s},\tilde{\x}) \le v_i\, , \; i\irg 13\, .
 \]
 In particular, $u_3\ge 0$ and $v_3 \ge 0$.
 We now verify that
 $$\lambda (u_{0},u_{1},u_{2},u_{3})+(1-\lambda) (v_{0},v_{1},v_{2},v_{3}) \in \bar \Omega.$$
Note that $\bar h_i(\y) = \y\T \left[ \Ob\oplus  \Qb_i \right] \y$ for $i\in \lk0,1\rk$ and $\bar h_2(\y) = \y\T \left[ \Ob\oplus  \Ab\T\Ab  \right] \y$, cf.~\eqref{defjo},
while
$$\bar h_3(\y) = \y\T \left[\begin{array}{cc}
                       \Ib_n &  -\Ib_n \\
                       -\Ib_n & \Ib_n
                       \end{array}
 \right] \y \, $$
 so that the associated matrices
\[
\left[\begin{array}{cc}
                       \Ob & \Ob\T \\
                       \Ob & \Qb_0
                       \end{array}
 \right],  \, \left[\begin{array}{cc}
                       \Ob & \Ob\T \\
                       \Ob & \Qb_1
                       \end{array}
 \right], \,  \left[\begin{array}{cc}
                       \Ob & \Ob\T \\
                       \Ob & \Ab\T\Ab
                       \end{array}
 \right], \left[\begin{array}{cc}
                       \Ib_n &  -\Ib_n \\
                       -\Ib_n & \Ib_n
                       \end{array}
 \right]
\]
are all $Z$-matrices. We see that
  \[
\{(\bar{h}_0(\y),\bar{h}_1(\y),\bar{h}_2(\y),\bar{h}_3(\y)): \y \in  \mathbb{R}^p \times 	\mathbb{R}^n \}+\mathbb{R}_+^4 \]
is convex.  So there exists $(\r,\z) \in \mathbb{R}^p \times \mathbb{R}^n$ such that
 \[
 \bar h_i(\r,\z) \le \lambda u_i+(1-\lambda)v_i \, ,\;  i\irg 03\, .
 \]
 Denote $\z=(z_1,\ldots,z_n)$ and let $|\z|=(|z_1|,\ldots,|z_n|)$. The $Z$-matrices assumptions ensure
 \[
 \bar{h}_i(|\z|,|\z|)=|\z|\T \Qb_i |\z|  \le  \z\T\Qb_i\z \le \lambda u_i +(1-\lambda) v_i\, , \; i\in \lk 0,1\rk\, ,
 \]
 \[
 \bar{h}_2(|\z|,|\z|)= |\z|\T (\Ab\T\Ab) |\z|  \le \z\T (\Ab\T\Ab)\z  \le \lambda u_2 +(1-\lambda) v_2
 \]
 and
 \[
 \bar{h}_3(|\z|,|\z|)=0 \le \lambda u_3 +(1-\lambda) v_3.
 \]
 Therefore, $\lambda (u_{0},u_{1},u_{2},u_{3})+(1-\lambda) (v_{0},v_{1},v_{2},v_{3}) \in \bar \Omega$, and so the conclusion follows.
 \end{proof}

 \begin{prop}
 Suppose that there exist $\tau_i\ge 0$, $i\irg 02$,  such that
 \[
 \tau_0 \Qb_0+ \tau_1\Qb_1 + \tau_2 \Ab\T\Ab \succ 0.
 \]
 Then $\Omega$ is closed.
 \end{prop}
 \begin{proof}
 Let $\r^{(k)} \in \Omega$ such that $\r^{(k)}\rightarrow\r\in \R^4$.
 Then there exists $\y_k=(\s_k,\x_k)\in \R^p_+\times \R^n$ such that
 \[
 \bar{f}_i(\y_k) \le  r_{i}^{(k)} \quad \mbox{for all } i\irg 03 \mbox{ and all }k\, .
 \]
 We first see that $\{\x_k\}$ is bounded.  To see this, note that
 $$ \sum_{i=0}^2 \tau_i f_i(\x_k) = \sum_{i=0}^2 \tau_i \bar f_i(\x_k) \le \sum_{i=0}^2 \tau _i r_i^{(k)} \rightarrow \sum_{i=0}^2 \tau _i r_i\, .$$
  Since $\nabla^2 \big(\sum_{i=0}^2 \tau_i f_i\big)(\x) \equiv \tau_0 \Qb_0+ \tau_1 \Qb_1+\tau_2 \Ab\T\Ab  \succ 0$, this implies that
 $\{\x_k\}$ must be bounded. Taking into account that
 $$\bar f_3(\s_k,\x_k)=\|\Bb\x_k+\s_k-\b\|^2 \le r_{3}^{(k)} \rightarrow r_3\, ,$$
 it follows that also $\{\s_k\}$ is a bounded sequence.
 By passing to subsequences, we may assume that $\y_k=(\s_k,\x_k) \rightarrow (\s,\x)=:\y\in \R^p_+\times \R^n$. Passing to the limit, we see that $\bar{f}_i(\y) \le  r_{i}, \ i\irg 03$ and so  $\r \in \Omega$. Thus $\Omega$ is closed.
 \end{proof}

\vspace{-0.2cm}
\subsection*{Sufficient global optimality conditions}
Now, we obtain the following sufficient second-order global optimality condition, which also implies that the copositive relaxation is tight, generalizing a recent result~\cite[Section~6.3]{Bomz13b} for
CDT problems:

\begin{thm}\label{tight}
If at a generalized KKT pair  $(\bar \x; \bar \u,\bar\v)\in (F\cap P)\times \R^3_+\times\R^p$ of problem~\reff{qqpp}, we have
\beq{defbars}\bar\Sb:= \Mb(L  (\cdot;\bar \u,\o ))-f_0(\bar \x)\Jb\in \CC^\star_\Upsilon\, ,\eeq
then $\bar\x$ is a globally optimal solution  to~\reff{qqpp} and $z^*=z_{\rm COP}^*$.
\end{thm}
\bep We first note that the conic dual of problem~\reff{copd} is
\beq{copp}
z_{\rm CP}^* = \inf\lk \du {\Mb_0} \Xb : \du{\Mb_i}\Xb \le 0, \, i\irg 1{3}\, , \, \du {\Jb} \Xb = 1 \, , \, \Xb\in \CC_\Upsilon\rk\, ,\eeq
with $\Mb_i = \Mb(\bar f_i)$ and $\CC_\Upsilon = \conv \lk \x\x\T : \x\in  \R^{p+1}_+\times \R^n \rk $.
From standard conical lifting and weak duality arguments it follows
$$z_{\rm LD}^*\le z_{\rm COP}^* \le z_{\rm CP}^*
\le z^*\, .$$
Let $\bar\s= \b-\Bb\bar\x$ and $\bar\y=(\bar \s, \bar\x)$. The complementarity conditions imply $\bar\v\T\bar\s = 0$ and $\sum\limits_{i=1}^{3}u_i \bar f_i(\bar \y)= 0$, so that both the standard $L (\bar\y;\bar\u,\bar \v) = f_0(\bar\x)$ and as well $L (\bar\y;\bar\u,\o) = f_0(\bar\x)$, which will be used now. Indeed, put
$\bar\z\T =[1,\bar\y\T]$ and $\bar\Xb = \bar\z\bar\z\T\in \CC_\Upsilon$. Then from the definition of $\bar \Sb$ we get
$$\du {\bar\Xb}{\bar\Sb} = \qf{\bar\z}{\bar\Sb} = L (\bar \y; \bar\u,\o) - f_0(\bar \x) = 0\, ,$$
so that $( \bar\Xb,\bar\Sb)$ form an optimal primal-dual pair for the copositive problem~\reff{copp} and~\reff{copd} with zero duality gap. We conclude, by feasibility of $\bar\x$ and definition of $z^*$, and because of~\reff{copd} with $\mu =f_0(\bar\x)$, cf.~\reff{defbars},
$$z^* \le  f_0(\bar\x)  \le z_{\rm COP}^* \le z_{\rm CP}^* \le  z^*$$
yielding tightness of the copositive relaxation, zero duality gap for the copositive-cp conic optimization problems, and optimality of $\bar\x$.\ep

While checking copositivity is NP-hard, the slack matrix $\bar \Sb$ may lie in a slightly smaller but tractable approximation cone, and then global optimality is guaranteed even in cases where $\bar \Sb$ is indefinite. The difference can also be expressed in properties of the Hessian $\Hb_{\bar \u}$ of the Lagrangian (recall that this is the same irrespective of our decision whether to relax also the linear constraints or not): indeed, a similar condition on the slack matrix yielding tightness of the classical Lagrangian bound (i.e. $z_{\rm LD}^*=z^*$) or the equivalent SDP relaxation~\cite[Section~5.1]{Bomz13c} implies that its lower right principal submatrix $\Hb_{\bar\u}$ has to be psd, and we know this is too strong in some cases~\cite{Yuan90}.

By contrast, the condition $\bar\Sb\in\CC^\star_\Upsilon$ {(giving tightness $z_{\rm COP}^*=z^*$)}, by the same argument using Lemma~\ref{schur} and~\reff{explsl}, only yields positive semidefiniteness of $\Hb_{\bar\u}+ \bar\u_{3}\Bb\Bb\T$. Of course, this happens with higher frequency than positive-definiteness of the Hessian, and the discrepancy is not negligible,
see~\cite[Section~5]{Bomz13c} for an example.

We note that Theorem \ref{tight} can be used to construct a class of problems where copositive relaxation is always tight while the usual Lagrangian dual produces a trivial bound with
infinite duality gaps. To see this, we shall need the following auxiliary result.

\begin{lemma}
Let $\Mb$ be strictly $\R^n_+$-copositive; then there exists a constant $\sigma>0$ such that
$$\qf\x\Mb + \sigma \norm{\x-\s}^2 >0 \mbox{ for all } (\s,\x)\in \big(\R^n_+\times\R^n\big)\setminus\lk \o\rk\, .$$
\end{lemma}

\bep  We first note that the conclusion trivially holds if $\Mb$ is further assumed to be positive semidefinite. So we may assume without loss of generality that
$\lambda_{\rm min} (\Mb) <0$.
For any $\x\in \R^n$, denote by
$$\x^+:=[\max\{0,x_1\},\ldots,\max\{0,x_n\}]\T \in \R^n_+$$ and by $\x^-:= \x^+-\x\in \R^n_+$ so that $\x = \x^+-\x^-$. Furthermore, we have $\norm{\x-\s}\ge\norm{\x^-}$ for all $\s\in \R^n_+$, as can be seen easily. Therefore we are done if we establish the (non-quadratic) inequality $\qf\x\Mb + \sigma \norm{\x^-}^2 >0$ whenever $\x\in \R^n\setminus\lk \o\rk$. Now, given $\Mb$ is strictly copositive, we choose $\rho :=\min \lk \qf\x\Mb : \x\in \R^n_+ , \norm \x =1\rk>0$. Note that $\x\in \R^n_+$ if and only if $\norm {\x^-}=0$. By continuity and compactness, we infer existence of an $\eps>0$ such that $\qf\x\Mb \ge\frac \rho 2$ whenever $\norm {\x^-}\le\eps$ and $\norm \x = 1$. Now we distinguish two cases:\\
Case 1: $\norm {\x^-}\ge \eps\norm \x>0$. In this case, we have
$$\qf\x\Mb \ge \lambda_{\rm min} (\Mb) {\norm\x}^2 > -\sigma {\norm {\x^-}}^2\, ,$$
where we set $\sigma := -  2\lambda_{\rm min} (\Mb)/\eps^2 >0$;

\noindent Case 2: $\norm {\x^-}\le \eps\norm \x$. In this case,  one has $\qf\x\Mb \ge \frac \rho 2 {\norm\x}^2 > -\sigma \norm{\x^-}^2$ (for any $\sigma>0$).

\noindent So in both cases, we obtain $\qf\x\Mb + \sigma{\norm{\x^-}}^2 >0$ for all $\x\in \R^n\setminus\lk\o\rk$, and the lemma is shown.
\ep
\vspace{-0.2cm}

\medskip

\noindent 
Consider the following non-convex quadratic optimization problem
\begin{eqnarray*}
\mathrm{(EP)} & \min & f_0(\x):= \x\T\Qb_0\x \\
& \mbox{subject to} & \|\x\|^2 \le 1\, , \\
& & \x \in \R^n_+ \, ,
\end{eqnarray*}
where $\Qb_0$ is a strictly $\mathbb{R}^n_+$-copositive and indefinite matrix. This problem can be regarded as a special case of the model problem~{\rm{(P)}} with
 $\Qb_1=\Ib_n$, $\Ab=\Ob$ and $\a=0$. We now see that this class of quadratic optimization admits a tight copositive relaxation and an infinite Lagrangian duality gap.
\begin{prop}\label{remark:3.2} {\bf (Tight copositive relaxation and infinite Lagrangian duality gap for (EP)\big)}
For problem (EP), let $z^*$, $z_{\rm LD}^*$ and $z_{\rm COP}^*$ denote the optimal value of (EP), the Lagrangian relaxation of (EP) and
 copositive relaxation of (EP) respectively. Then  $z_{\rm LD}^*=-\infty$ and $z^*=z_{\rm COP}^*=0$.
\end{prop}
  \begin{proof}
  Direct verification shows that $\o \in \mathbb{R}^n$ is a global solution with the optimal value $z^*=0$. We first observe that, as $\Qb_0$ is indefinite, the optimal value of the Lagrangian dual is $z_{\rm LD}^*=-\infty$.
Next, as $\Qb_0$ is strictly $\mathbb{R}^n_+$-copositive, the preceding lemma implies that there exists $\sigma>0$ such that
\[
\x\T\Qb_0\x + \sigma \norm{\x-\s}^2 >0 \mbox{ for all } (\s,\x)\in \big(\R^n_+\times\R^n\big)\setminus\lk \o\rk \, .
\]
Let $\bar \y=(\bar \x,\bar \s)$ with $\bar \x=\bar \s=\o \in \mathbb{R}^n$, $\bar \u=(0,0,\sigma) \in \mathbb{R}^3_+$ and $\bar \v=\o \in \mathbb{R}^n$. Then we see that
$(\bar \y;\bar \u,\bar \v)$ is a (generalized) KKT pair for (EP). Moreover, we have
\[
\Mb(L (\cdot;\bar \u,\o))-f_0(\bar \x) \Jb=\Mb(L (\cdot;\u,\o))=\left[\begin{array}{ccc}
0 &\o\T & \o\T  \\
\o& u_{3}\Ib_n & -u_{3}\Ib_n \\
\o   & -u_{3}\Ib_n & \Qb_0 + u_{3}\Ib_n
\end{array} \right ],
\]
For all $\d:=(r,\s,\x)\in \Upsilon = \mathbb{R}_+ \times \R^n_+\times\R^n$, above implies
\[
\d\T\left [ \Mb(L (\cdot;\bar \u,\o))-f_0(\bar \x) \Jb\right ] \d= \x\T\Qb_0\x+ \sigma\|\x-\s\|^2 \ge 0\, ,
\]
and hence $\Mb(L (\cdot;\bar \u,\o))-f_0(\bar \x) \Jb$ is $\Upsilon$-copositive. This shows that $z_{\rm COP}^*=0$.
\end{proof}


\section{Relaxation tightness in extended CDT Problems}
In this section, we examine the so-called extended CDT problem:
\begin{eqnarray*}
\mathrm{(P_{\rm CDT})}
 & \min & \x\T\Qb_0 \x+2 \q_0\T\x \\
& \mbox{ subject to } &  \|\x\|^2 \le 1 \\
& & {\norm {\Ab\x-\a}}^2 \le 1 \\
& & \Bb\x \le \b.
\end{eqnarray*}
This problem is a special case of our general model problem with $\Qb_1=\Ib_n$ and $\q_1=\o$. In the cases where no linear inequalities are present, the problem $\mathrm{(P_{CDT})}$ reduces to the so-called CDT problem (also referred as two-ball trust region problems, TTR). The CDT problem, in general, is much more challenging than the well-studied trust region problems and has received much attention lately, see for example \cite{AiZh09,Beck06,Bienstock,Bomz13b,Bure13a}. The problem $\mathrm{(P_{CDT})}$ arises from robust
optimization \cite{Jeya13a} as well as applying trust region techniques for solving nonlinear optimization problems with both nonlinear and linear constraints: see \cite{Powell} for the case of trust region problems with additional linear inequalities and
see \cite{Bienstock,Bomz13b} for the case of CDT problems. We will establish simple conditions ensuring exactness of the copositive relaxations and the usual Lagrangian relaxations of the extended CDT problem.

First of all, we note that the sufficient second-order global optimality condition in Theorem~\ref{tight}, specialized to the setting $\mathrm{(P_{CDT})}$,
yields the exact copositive relaxation for extended CDT problems.
\begin{cor}
Let $(\bar \x; \bar \u,\bar\v)\in F_{\rm CDT} \times \R^3_+\times\R^p$ be a generalized KKT pair of problem $\mathrm{(P_{CDT})}$ where $F_{\rm CDT}$ is the feasible
set of $\mathrm{(P_{CDT})}$. Denote by $\bar \mu:=\bar \x\T\Qb_0 \bar \x+2 \q_0\T \bar \x$. Suppose that
\[
\left[\bea{ccc}
-\bar u_1-\bar u_2 + \bar u_{3}{\norm \b}^2-\bar \mu &- \bar u_{3}\b\T & \q_0\T-\bar u_2 \a\T \Ab - \bar u_{3}\b\T\Bb \\
- \bar u_{3}\b & \bar u_{3}\Ib_p & \bar u_{3}\Bb \\
\q_0-\bar u_2  \Ab\T \a - \bar u_{3}\Bb\T \b  & \bar u_{3}\Bb\T & \Qb_0 + \bar u_1 \Ib_n+ \bar u_2 \Ab\T\Ab+ \bar u_{3}\Bb\T\Bb
\ea
\right ]
\]
is $(\mathbb{R}^{p+1}_+ \times \mathbb{R}^n)${-copositive}. Then  $\bar\x$ is a globally optimal solution  to problem~$\mathrm{(P_{CDT})}$ and $z^*=z_{\rm COP}^*$.
\end{cor}
\begin{proof}
The conclusion follows by Theorem \ref{tight} with  $\Qb_1=\Ib_n$ and $\q_1=\o$.
\end{proof}

\vspace{-0.2cm}
Next we examine when the usual Lagrangian relaxation is exact for the extended CDT problems. To this end, we define an auxiliary convex optimization problem
\begin{eqnarray*}
\mathrm{(AP)} & \min & \x\T \Qb_0^+\x+2 \q_0\T\x \\
& \mbox{ subject to } & \|\x\|^2\le 1 \\
& & {\norm {\Ab\x-\a}}^2 \le 1 \\
& & \Bb\x \le \b\, ,
\end{eqnarray*}
where
\beq{defqoplus}
\Qb_0^+:= \Qb_0-\lambda_{\min}(\Qb_0)\Ib_n \succeq \Ob\, .\eeq
We first see that if the auxiliary convex problem (AP) has a minimizer on the sphere $\{\x\in \R^n :\|\x\|=1\}$, then an extended CDT problem has a tight semi-Lagrangian
relaxation. We will provide a sufficient condition in terms of the original data guaranteeing this condition later (in Theorem 4.1).
\begin{lemma}
Suppose that the auxiliary convex problem (AP) has a minimizer on the sphere $\{\x:\|\x\|=1\}$. Then $z_{\rm LD}^*=z_{\rm COP}^*=z^*$.
\end{lemma}
\begin{proof}
Recall that $z_{\rm LD}^* \le z_{\rm COP}^* \le z^*$. So it suffices to show that $z_{\rm LD}^*=z^*$.
Without loss of generality, we assume that $\lambda_{\min}(\Qb_0)<0$ (otherwise $\mathrm{(P_{CDT})}$ is a convex quadratic problem and so $z_{\rm LD}^*=z^*$).
Let $\x^*$ be a solution of~(AP) with $\|\x^*\|=1$. As $\lambda_{\min}(\Qb_0)<0$, it follows from $\norm {\x} \le 1$ for all $\x\in F_{\rm CDT}$ that
\begin{eqnarray*}
z^* & = & \min_{\x\in F_{\rm CDT}} f_0(\x)\, \ge \, \min\limits_{\x\in F_{\rm CDT}} [f_0(\x) + \lambda_{\min}(\Qb_0)(1-\norm{\x}^2)] \\
 &= & \min\limits_{\x\in F_{\rm CDT}} \{\x\T \Qb_0^+\x+2 \q_0\T\x \} + \lambda_{\min}(\Qb_0)  \\
& = & {\x^*}\T \Qb_0^+ \x^* +2 \q_0\T\x^*+\lambda_{\min}(\Qb_0) \\
& = &  {\x^*}\T \Qb_0^+ \x^* +2 \q_0\T\x^*+\lambda_{\min}(\Qb_0)\|\x^*\|^2 \\
& = & {\x^*}\T\Qb_0 \x^* +2 \q_0\T\x^* \, = \; f_0(\x^*) \,\ge\, z^* \, ,
\end{eqnarray*}
where the last inequality follows from feasibility of $\x^*$ for the extended CDT problem. This shows that $z^*={\rm val}\mathrm{(AP)}+\lambda_{\min}(\Qb_0)$. Rewriting (AP) as
\begin{eqnarray*}
\mathrm{(AP1)}  & \min\limits_{(\x,\s) \in \mathbb{R}^n \times \mathbb{R}^p_+} & \x\T \Qb_0^+ \x+2 \q_0\T\x \\
& \mbox{ subject to } & \|\x\|^2\le 1 \\
& & {\norm {\Ab\x-\a}}^2 \le 1 \\
& & \|\Bb\x+\s-\b\|^2  =0\, ,
\end{eqnarray*}
we obtain the Lagrangian dual of this problem which can be stated as
\begin{eqnarray*}
\mathrm{(LD1)} & & \sup\limits_{(u_3,u_1,u_2,\v) \in \mathbb{R}\times\mathbb{R}^{p+2}_+}\quad\inf\limits_{(\x,\s) \in \mathbb{R}^n \times \mathbb{R}^p}\{ \x\T \Qb_0^+ \x+2 \q_0\T\x  +u_1f_1(\x) \\[-0.5em]
& & \ \ \ \ \ \ \ \ \ \ \ \ \ \ \ \ \ \ \ \ \ \ \ +u_2f_2(\x)+u_3 (\|\Bb\x+\s-\b\|^2)-2\v\T\s\} \\[0.8em]
& =& \sup\limits_{(\mu,u_3,u_1,u_2,\v) \in \mathbb{R^2}\times\mathbb{R}^{p+2}_+}\{\mu: \tilde \Mb(\mu,\u,\v) \succeq \Ob\}\, ,\end{eqnarray*}
where $\tilde \Mb(\mu,\u,\v)$ denotes the matrix
$$\left[\bea{ccc}
-u_1-u_2 + u_{3}{\norm \b}^2-\mu &- u_{3}\b\T-\v\T & -\q_0\T-u_2 \a\T\Ab - u_{3}\b \T \Bb \\
 - u_{3}\b-\v & u_{3}\Ib_p & u_{3}\Bb \\
 -\q_0-u_2 \Ab\T\a - u_{3}\Bb\T\b & u_{3}\Bb\T & \Qb_0^+ + u_1 \Ib_n+u_2 \Ab \T  \Ab + u_{3}\Bb \T  \Bb
 \ea
 \right ].
$$
Note that the feasible set of (AP1) is bounded by $\|\x\|\le 1$ and $\s=\Bb\x-\b$ for all feasible $(\x,\s)$. Since any convex optimization problem
with compact feasible set enjoys a zero duality gap (for example see \cite{Jeya90}), it follows that
\[
{\rm val}\mathrm{(AP)}={\rm val}\mathrm{(LD1)}\, .
\]
Finally, the conclusion results by noting that ${\rm val}\mathrm{(LD1)}=z_{\rm LD}^*+\lambda_{\min}(\Qb_0)$. So we have $z^*=z_{\rm LD}^*$ and furthermore
$z_{\rm LD}^*=z_{\rm COP}^*=z^*$.
\end{proof}

Next we provide a simple sufficient condition formulated in terms of the original data guaranteeing tightness of the relaxations. It is important to note that this sufficient condition can be efficiently verified by solving a feasibility problem of a linear optimization problem.
\begin{thm}
Let $\Mb=[\Qb_0^+ | \Ab\T]\T$. Suppose that \begin{equation}\label{LP_cond}
{\rm ker}(\Mb) \cap \{\v \in \mathbb{R}^n: \Bb \v\le \o\} \cap \{\v: \q_0\T\v \ge 0\} \neq \{\o\}\, .
\end{equation}
Then  $z_{\rm LD}^*=z_{\rm COP}^*=z^*$.
\end{thm}
\begin{proof}
By the preceding lemma, the conclusion follows if we show that the auxiliary convex problem (AP) has a minimizer on the sphere $\{\x:\|\x\|=1\}$. Suppose that a minimizer $\x^*$ of (AP) satisfies $\|\x^*\|<1$. Let $\v \in {\rm ker}(\Mb) \cap \{\v \in \mathbb{R}^n: \Bb \v\le \o\} \cap \{\v: \q_0\T\v \ge 0\}$ with $\v \neq \o$. Consider $\x(t)=\x^*+t\v$, $t \ge 0$. Then  there exists $t_0>0$ such that $\|\x(t_0)\|=1$. Now observe
\begin{eqnarray*}
& & \x(t_0)\T \Qb_0^+ \x(t_0)+2 \q_0\T\x(t_0) \\[0.2em]
&=& {\x^*}\T \Qb_0^+ \x^*+2 \q_0\T\x^* -2t_0 \q_0\T\v
\\[0.2em]
&\le &{\x^*}\T \Qb_0^+ \x^*+2 \q_0\T\x^*\, ,
\end{eqnarray*}
\[
{\norm {\Ab\x(t_0)-\a}}^2={\norm {\Ab\x^*-\a}}^2 \le 1\, ,
\]
and
\[
\Bb\x(t_0)-\b  =\Bb\x^*-\b +t_0\Bb \v \le \o\,.
\]
This shows that $\x(t_0)$ is a minimizer for (AP) and $\|\x(t_0)\|=1$. The conclusion follows.
\end{proof}
\vspace{-0.3cm}
\begin{rem}[LP reformulation of the sufficient condition \eqref{LP_cond}]
Our sufficient condition \eqref{LP_cond} can be efficiently verified by determining a feasible solution to the following LP
\begin{eqnarray*}
&\inf\limits_{\d \in \mathbb{R}^n}\{1:\Mb \d =0, \, \Bb\d \le 0, \, -\q_0^Td \le 0, \, \sum_{i=1}^nd_i=1\}.
\end{eqnarray*}
\end{rem}
\begin{rem}{\bf (Links to the known dimension condition for exact relaxation)}
In the special case where $\Ab=\Ob$ and $\a =\o\in \mathbb{R}^n$, the authors showed in \cite{Jeya13a}, that under the dimension condition
$${\rm dim}\, {\rm ker} \Qb_0^+ \ge {\rm dim} \, {\rm span}[\b_1,\ldots,\b_p]+1\, ,$$
where $\b_i\T$ is the $i$th row of $\Bb$, the SDP (or Lagrangian) relaxation is exact.
We observe that this dimension condition is strictly stronger than our sufficient condition in the preceding theorem.

Firstly, we see that the dimension condition implies
our sufficient condition in the preceding theorem. To see this, suppose the above dimension condition holds. Then there exists
$\v \neq \o$ such that $\v \in {\rm ker} \Qb_0^+$ and $\b_i\T \v=0$ for all $i\irg 1p$ (and hence $\Bb \v=\o$). By replacing
$\v$ by $-\v$ if necessary, we can assume that $\q_0\T\v \ge 0$. Thus, our sufficient condition in the preceding theorem holds.

To see the dimension condition is strictly stronger, let us consider $\Qb_0=\left[\begin{array}{cc}
2 & 0 \\
0 & -2
\end{array}\right]$, $\Ab=\left[\begin{array}{cc}
0 & 0 \\
0 & 0
\end{array}\right]$, $\q_0=\a= \left[\begin{array}{c} 0\\0\end{array}\right]$  and $\Bb=\b_1\T=[-1,0]$. Clearly, ${\rm dim}\, {\rm ker} \Qb_0^+={\rm dim}\, {\rm ker} \left[\begin{array}{cc}
4 & 0 \\
0 & 0
\end{array}\right]=1$ and ${\rm dim} \, {\rm span}[\b_1]=1$, and so the dimension condition fails. On the other hand, our sufficient condition reads
$$\lk \o \rk \neq {\rm ker} \left[\begin{array}{cccc}
4 & 0 & 0 & 0\\
0 & 0 & 0 & 0
\end{array}\right]\T \cap \{\v \in \mathbb{R}^2: -v_1 \le 0\} = \lk 0\rk\times\R \, ,$$
which is obviously satisfied.
\end{rem}

\section{Approximation hierarchies for $\CC_\Upsilon$-copositivity}\label{sos-sec}

In general, checking copositivity of  a matrix is an NP-hard problem, and hence solving a copositive optimization problem is also NP-hard. Therefore, to compute the semi-Lagrangian, we need to approximate them by so-called {\em hierarchies,} i.e., a sequence of conic optimization problems involving tractable cones $\DD_d^\star$ such that $\DD_d^\star\subset \DD_{d+1}^\star\subset \CC_{\Upsilon}^\star$ where $d$ is the level of the hierarchy, and ${\rm cl}(\bigcup_{d=0}^\infty \DD_d^\star) = \CC_{{\Upsilon}}^\star$. On the dual side,
 $\DD_d$ are also tractable, $\DD_{d+1} \subset \DD_d$, and $\bigcap_{d=0}^\infty \DD_d= \CC_{\Upsilon}$. For classical copositivity $\CC_{\R^n_+}$, there are many options, for a concise survey see~\cite{Bomz12x}. Many of these hierarchies involve linear~ \cite{Bund08,Bund09a} or psd constraints of matrices of order $n^{d+2}$, e.g. the seminal ones proposed in~\cite{Lass00,Parr00a}.
One possibility would be the reduction of $\Upsilon$-copositivity via Schur complements as in Lemma~\ref{schur} above, reducing this question to a combination of psd and classical copositivity conditions, which can be treated by these classical approximation hierarchies. However, the difficulty with this approach is the nonlinear dependence of the Schur complement $\Rb- \Sb\T\Hb^\dag \Sb$ on $(\mu, \u)$. Therefore let us outline two alternative approaches for constructing tractable hierarchies in approximating the copositive relaxation, extending and adapting the classical approach.

\noindent \textbf{SDP Hierarchy.}
One approach for computing the copositive relaxation is to use a hierarchy of SDP relaxations,  extending the sum-of-squares idea in Parrilo's work~\cite{Parr00a}, which we sketch as below.
Let $I=[1\!:\!p+1]$ and, for a matrix $\Mb =\Mb\T\in \mathbb{R}^{(p+n+1)\times (p+n+1)}$, define a quartic polynomial
\[
p_\Mb(\y)=\sum_{(i,j) \in I \times I} M_{ij}y_i^2 y_j^2 + \sum_{i \in I, j \notin I} M_{ij}y_i^2 y_j+ \sum_{i \notin I, j \in I} M_{ij}y_i y_j^2+ \sum_{i \notin I, j \notin I} M_{ij}y_i y_j\, .
\]
Note that $\Mb$ is $\Upsilon${-copositive} if and only if  $p_\Mb(\y) \ge 0$ for all  $\y \in \mathbb{R}^{p+n+1}$. A sufficient condition for this is that the product $p_\Mb(\y){\norm\y}^{2d} = p_\Mb(\y)(\sum_k y_k^2)^d$ is a \emph{sum-of-squares (s.o.s.)} polynomial, which automatically guarantees nonnegativity of $p_\Mb$ over $\R^{p+n+1}$.
Now it is natural to define
\begin{equation}\label{eq:SDP}
\DD_d^\star := \lk \Mb = \Mb\T \in \R^{(p+n+1)\times (p+n+1)} : p_\Mb(\y) {\norm \y}^{2d}\mbox{ is a s.o.s. polynomial}\rk\, .
\end{equation}
Then, following the logic of classical hierarchies, it is not difficult to see that above properties hold, and that $\DD_d^\star$ (and $\DD_d$ itself) are tractable cones expressible by LMI conditions on matrices of order $n^{d+2}$. Thus, copositivity characterization of the semi-Lagrangian relaxation can be computed by using SDP hierarchies and polynomial optimization techniques~\cite{Parr00a}. Of course, LMIs on matrices of larger order pose a serious memory problem for algorithmic implementations even for moderate $d$ if $n$ is large.

However, in recent years, various techniques have been proposed to address this issue: one approach is to exploit special structures of the problem such as sparsity and symmetry \cite{kim2,JLLL} to treat large scale polynomial problems. Other techniques include refined SDP hierarchies such as the SDP approximation proposed in \cite{Lassere_COP} and the recently established bounded s.o.s. hierarchy \cite{Bounded_SOS}.

On the other hand, it is worth noting that sometimes even the zero-level approximation in the hierarchy \eqref{eq:SDP} can provide a much better bound as compared to the Lagrangian relaxation, as shown in the next example.

\begin{example} {\bf (Zero-level approximation of copositive relaxations can beat the Lagrangian relaxation)}
With the data from Example~\ref{ex:5.1}, recalling that the optimal value of this example is $z^*=0$, we have
\[
\widehat{\Mb}(u)=\left[\bea{ccccc}
 u & 0 & -u & 0 \\
 0 & u & 0 & -u \\
 -u & 0 & \frac{1}{2}+u & 1 \\
 0 & -u & 1 & 1+u
\ea \right ] \quad\mbox{ and }\quad \Mb(u,\mu)=\left[\bea{cc}
 -\mu & \o\T \\
\o & \widehat{\Mb}(u)
\ea
\right ]\, .
\]
Then the zero-level approximation for the copositive relaxation problem reads
\begin{eqnarray*}
& \mathrm{(RP)} & \sup_{(\mu,u) \in \R\times\R_+}\lk  \mu : p_{\Mb(u,\mu)} \mbox{ is a s.o.s. polynomial\/} \rk\, .
\end{eqnarray*}
From the definition of $\Mb(u,\mu)$, we see that $p_{M(u,\mu)}$ is a s.o.s. polynomial if and only if
$\mu \le 0$ and
$$\widehat{p}_u(\x)=\left[x_1^2, x_2^2, x_3,x_4\right] \widehat{\Mb}(u)\left[\begin{array}{c}
x_1^2\\ x_2^2\\ x_3 \\ x_4
\end{array}\right ]$$
is a s.o.s. polynomial.
This shows that
\[
{\rm val}({\mathrm{RP}})= \left\{\begin{array}{cll}
0, & \mbox{ if } & C \neq \emptyset, \\
-\infty, & \mbox{ else}.
\end{array}\right.
\]
where
\[
C:=\{u \in \mathbb{R}: \widehat{p}_u \mbox{ is a s.o.s. polynomial\/}\}\,.
\]
Using the ``solvesos" command in the Matlab toolbox YALMIP \cite{Loef09}, one can verify that $\hat{p}_u$ is a s.o.s. polynomial if $u=1$. Thus,
${\rm val}\mathrm{(RP)}=0$ which agrees with the true optimal value $z^*=0$.
 On the other hand, as computed in Example~\ref{ex:5.1}, the Lagrangian relaxation yields a trivial lower bound $-\infty$.
\end{example}

\noindent\textbf{LP Hierarchy.}
Another approach is to compute the copositive relaxation using linear optimization.
While providing, in general, weaker bounds in comparing with SDP hierarchies, this approach is appealing because LP solvers suffer less from memory problems than
 state-of-art SDP solvers. To do this, consider a compact polyhedral base
$K$ of the cone $\Upsilon$, i.e., $\R_+K =\Upsilon$, e.g. the polytope
$$K = \lk (\s,\x)\in \Upsilon : \sum_{i=1}^{p+1}s_i\le 1\,, \max_{i\in[1:n]}  |x_i| \le 1\rk = \lk \y \in \R^{p+n+1} : h_j(\y) \ge 0 \, ,\; j\irg 1m\rk$$
described by $m=p+2(n+1)$ affine-linear inequalities. By positive homogeneity,
%
we observe that $\Mb$ is $\Upsilon${-copositive} if and only if  $q_\Mb(\y):=\y^T \Mb \y \ge 0$ for all  $\y \in K$. Now
Handelman's theorem (for example see \cite[Theorem 2.24]{Lassere_book}) ensures that any polynomial $f$ positive over such a  polytope $K$ admits the representation $f=\sum_{{\bf \alpha} \in \mathbb{N}^{m}}c_{{\bf \alpha}}\prod_{j=1}^{m} h_j^{\alpha_j}$ for some scalars $c_{{\bf \alpha}}\ge 0$.
Then one can construct a sequence of LP approximation by letting
$$\DD_d^\star := \lk \Mb = \Mb\T \in \R^{(p+n+1)\times (p+n+1)} : q_\Mb=\sum_{{\bf \alpha} \in \mathbb{N}^{m}, |{\bf \alpha}| \le d}c_{{\bf \alpha}}\prod_{j=1}^{m} h_j^{\alpha_j}, \, c_{\bf \alpha} \ge 0 \rk\, .$$
It is well known (see for example \cite[Theorem 5.11]{Lassere_book}) that above $\DD_d^\star$ can be expressed by linear inequality constraints.

\frenchspacing
\small
\addcontentsline{toc}{section}{References}
\bibliographystyle{plain}

\begin{thebibliography}{10}

\bibitem{AiZh09}
Wenbo Ai and Shuzhong Zhang.
\newblock Strong duality for the {CDT} subproblem: a~necessary and sufficient
  condition.
\newblock {\em SIAM J. Optim.}, 19(4):1735–--1756, 2009.

\bibitem{Amar15}
Paula~A. Amaral and Immanuel~M. Bomze.
\newblock Copositivity-based approximations for mixed-integer fractional
  quadratic optimization.
\newblock {\em Pacific J. Optimiz.}, 11(2):225--238, 2015.

\bibitem{Amar11}
Paula~A. Amaral, Immanuel~M. Bomze, and Joaquim~J. J\'{u}dice.
\newblock Copositivity and constrained fractional quadratic problems.
\newblock {\em Math. Program.}, 146(1--2):325--350, 2014.

\bibitem{Arde16a}
Amir Ardestani-Jaafari and Erick Delage.
\newblock Linearized robust counterparts of two-stage robust optimization
  problems with applications in operations management.
\newblock Preprint, {HEC Montr\'eal},
  http://www.optimization-online.org/DB\_HTML/2016/01/5388.html, 2016.

\bibitem{Beck06}
Amir Beck and Yonina Eldar.
\newblock Strong duality in nonconvex quadratic optimization with two quadratic
  constraints.
\newblock {\em SIAM J. Optim.}, 17(3):844–--860, 2006.


\bibitem{Bienstock}  Daniel Bienstock,
\newblock A note on polynomial solvability of the CDT problem.
\newblock {\em SIAM J. Optim.} 26(1):488-498,  2016.


\bibitem{Bomz13c}
Immanuel~M. Bomze.
\newblock Copositive relaxation beats {L}agrangian dual bounds in quadratically
  and linearly constrained {QP}s.
\newblock {\em SIAM J. Optim.}, 25(3):1249–--1275, 2015.

\bibitem{Bomz12x}
Immanuel~M. Bomze, Mirjam D\"ur, and Chung-Piaw Teo.
\newblock Copositive optimization.
\newblock {\em Optima -- {MOS N}ewsletter}, 89:2--10, 2012.

\bibitem{Bomz13b}
Immanuel~M. Bomze and Michael~L. Overton.
\newblock Narrowing the difficulty gap for the {C}elis-{D}ennis-{T}apia
  problem.
\newblock {\em Math. Program. Series B}, 151(2):459--476, 2015.

\bibitem{Bomz14b}
Immanuel~M. Bomze, Werner Schachinger, and Reinhard Ullrich.
\newblock New lower bounds and asymptotics for the cp-rank.
\newblock {\em SIAM J. Matrix Anal. Appl.}, 36(1):20--37, 2015.

\bibitem{Bund08}
Stefan Bundfuss and Mirjam D{\"u}r.
\newblock Algorithmic copositivity detection by simplicial partition.
\newblock {\em Linear Algebra Appl.}, 428(7):1511--1523, 2008.

\bibitem{Bund09a}
Stefan Bundfuss and Mirjam D{\"u}r.
\newblock An adaptive linear approximation algorithm for copositive programs.
\newblock {\em SIAM J. Optim.}, 20(1):30--53, 2009.

\bibitem{Bure09a}
Samuel Burer.
\newblock On the copositive representation of binary and continuous nonconvex
  quadratic programs.
\newblock {\em Math. Program.}, 120(2, Ser. A):479--495, 2009.

\bibitem{Bure13a}
Samuel Burer and Kurt Anstreicher.
\newblock Second-order-cone constraints for extended trust-region subproblems.
\newblock {\em SIAM J. Optim.}, 23(1):432–--451, 2013.

\bibitem{Conn00}
Andrew~R. Conn, Nicholas I.~M. Gould, and Philippe~L. Toint.
\newblock {\em Trust-region methods}.
\newblock MPS/SIAM Series on Optimization. Society for Industrial and Applied
  Mathematics (SIAM), Philadelphia, PA, 2000.

\bibitem{Dick11b}
Peter J.~C. Dickinson and Luuk Gijben.
\newblock On the computational complexity of membership problems for the
  completely positive cone and its dual.
\newblock {\em Comput. Optim. Appl.}, 57(2):403--415, 2014.

\bibitem{Faye07}
Alain Faye and Fr{\'e}d{\'e}ric Roupin.
\newblock Partial {L}agrangian relaxation for general quadratic programming.
\newblock {\em 4OR}, 5:75–--88, 2007.

\bibitem{Hana16a}
Grani~A. Hanasusanto and Daniel Kuhn.
\newblock Conic programming reformulations of two-stage distributionally robust
  linear programs over {W}asserstein balls.
\newblock Preprint, {EPFL},
 http://www.optimization-online.org/DB\_HTML/2016/01/5647.html, 2016.

\bibitem{Jeya09a}
Vaithilingam Jeyakumar, Gue~{M}yung Lee, and Guoyin Li.
\newblock Alternative theorems for quadratic inequality systems and global
  quadratic optimization.
\newblock {\em SIAM J. Optim.}, 20(2):983–--1001, 2009.

\bibitem{Jeya13a}
Vaithilingam Jeyakumar and Guoyin Li.
\newblock Trust-region problems with linear inequality constraints: exact {SDP}
  relaxation, global optimality and robust optimization.
\newblock {\em Math. Program.}, 147(Ser. A):171--206, 2014.

\bibitem{Jeya90}
Vaithilingam Jeyakumar and Henry Wolkowicz.
\newblock Zero duality gaps in infinite-dimensional programming.
\newblock {\em J. Optim. Theory Appl.}, 67(1):87--108, 1990.

\bibitem{JLLL}
Vaithilingam Jeyakumar, Sunyoung Kim, Gue Myung Lee, and Guoyin Li,
\newblock Solving global optimization problems with sparse polynomials and unbounded semialgebraic feasible sets,
\newblock {\em J. Global Optim.},  65:175-190, 2016.

\bibitem{kim2} Sunyoung Kim and Masakazu Kojima,
\newblock Exploiting sparsity in SDP relaxation of polynomial optimization problems,
\newblock Handbook on Semidefinite, Conic and Polynomial Optimization: theory, algorithm, software and applications, M. Anjos and J.B. Lasserre eds., 499-532, 2011.


\bibitem{Lass00}
Jean~Bernard Lasserre.
\newblock Global optimization with polynomials and the problem of moments.
\newblock {\em SIAM J. Optim.}, 11(3):796--817, 2000/01.

\bibitem{Lassere_book}
Jean~Bernard Lasserre.
\newblock Moments, Positive Polynomials and Their Applications
\newblock
Imperial College Press, World Scientific, Singapore, 2010.

\bibitem{Lassere_COP}
Jean~Bernard Lasserre.
\newblock New approximations for the cone of copositive matrices and its dual
\newblock  {\em Math. Program.} 144:265-276, 2014.

\bibitem{Bounded_SOS}    Jean~Bernard Lasserre, Kimchuan Toh and Shouguang Yang,
\newblock A bounded  degree SOS hierarchy for polynomial optimization
\newblock {\em European J. Comput. Optim.} 5:87--117, 2017.

\bibitem{Locatelli} Marco Locatelli,
\newblock Exactness conditions for an SDP relaxation of the extended trust region problem.
\newblock {\em Optim. Lett.} 10(6):1141--1151,  2016.

\bibitem{Loef09}
Johan L\"{o}fberg.
\newblock Pre- and post-processing sum-of-squares programs in practice.
\newblock {\em IEEE Trans. Automatic Control}, 54:1007--1011, 2009.

\bibitem{Motz52b}
Theodore~S. Motzkin.
\newblock Copositive quadratic forms.
\newblock Projects and Publications of the National Applied Mathematics
  Laboratories, Quarterly Report, April through June 1952, pp. 11--12, No.
  1818, National {B}ureau of {S}tandards, 1952.

\bibitem{Murt87}
Katta~G. Murty and Santosh~N. Kabadi.
\newblock Some {NP}-complete problems in quadratic and nonlinear programming.
\newblock {\em Math. Program.}, 39(2):117--129, 1987.

\bibitem{Nata11}
Kathik Natarajan, Chung~Piaw Teo, and Zhichao Zheng.
\newblock Mixed zero-one linear programs under objective uncertainty: a
  completely positive representation.
\newblock {\em Operations Research}, 59(3):713--728, 2011.

\bibitem{Parr00a}
Pablo~A. Parrilo.
\newblock {\em Structured Semidefinite Programs and Semi-algebraic Geometry
  Methods in Robustness and Optimization}.
\newblock PhD thesis, California Institute of Technology, Pasadena, CA, May
  2000.


\bibitem{Powell} Michael~J. D. Powell,
\newblock On fast trust region methods for quadratic models with linear constraints
\newblock   {\em Math. Program. Compu.}, 7 (3): 237-267, 2015.


\bibitem{Shak13b}
Naomi Shaked-Monderer, Abraham Berman, Immanuel~M. Bomze, Florian Jarre, and
  Werner Schachinger.
\newblock New results on the cp rank and related properties of
  co(mpletely~)positive matrices.
\newblock {\em Linear Multilinear Algebra}, 63(2):384--–396, 2015.

\bibitem{Shor87}
Naum~Z. Shor.
\newblock Quadratic optimization problems.
\newblock {\em Izv. Akad. Nauk SSSR Tekhn. Kibernet.}, 222(1):128--139, 1987.

\bibitem{Xu16a}
Guanglin Xu and Samuel Burer.
\newblock A copositive approach for two-stage adjustable robust optimization
  with uncertain right-hand sides.
\newblock Preprint, {Univ.} of {Iowa},
http://www.optimization-online.org/DB\_HTML/2016/01/5646.html, 2016.

\bibitem{Yang13a}
Boshi Yang and Samuel Burer.
\newblock A two-variable analysis of the two-trust-region subproblem.
\newblock {\em SIAM J. Optim.}, 26(1):661–--680, 2016.

\bibitem{Yuan90}
Ya-Xiang Yuan.
\newblock On a subproblem of trust region algorithms for constrained
  optimization.
\newblock {\em Math. Program.}, 47(1, (Ser. A)):53--63, 1990.

\end{thebibliography}

\end{document}